\documentclass{article}

\usepackage{amssymb}
\usepackage{amsmath}
\usepackage[all]{xy}

\DeclareMathOperator{\Hom}{Hom}

\DeclareMathOperator{\Mod}{Mod}

\DeclareMathOperator{\Tr}{Tr}
\DeclareMathOperator{\ssmod}{SSMod}
\DeclareMathOperator{\Soc}{Soc}

\DeclareMathOperator{\lep}{lep}
\DeclareMathOperator{\fil}{fil}

\DeclareMathOperator{\idom}{\mathfrak{In}}
\DeclareMathOperator{\pdom}{\mathfrak{Pr}}
\DeclareMathOperator{\ext}{Ext}
\DeclareMathOperator{\wis}{Wis}

\newtheorem{theorem}{Theorem}[section]
\newtheorem{proposition}[theorem]{Proposition}
\newtheorem{definition}[theorem]{Definition}
\newtheorem{corollary}[theorem]{Corollary}
\newtheorem{lemma}[theorem]{Lemma}

\newtheorem{example}[theorem]{Example}

\newenvironment{proof}[1][\it{Proof}]{\textbf{#1. } }{$\square$}

\begin{document}

\title{Characterizing rings in terms of the extent of the injectivity and projectivity of their modules.}
\author{Sergio R. L\'opez-Permouth and Jos\'e E. Simental}
\date{\today}

\maketitle

\begin{abstract}
Given a ring $R$, we define its right $i$-profile (resp. right $p$-profile) to be the collection of injectivity domains (resp. projectivity domains) of its right $R$-modules. We study the lattice theoretic properties of these profiles and consider ways in which properties of the profiles may determine the structure of rings and viceversa. We show that the $i$-profile is isomorphic to an interval of the lattice of linear filters of right ideals of $R$, and is therefore modular and coatomic. In particular, we give a practical characterization of the profile of a right artinian ring. We show through an example that the $p$-profile of a ring is not necessarily a set, and also characterize the $p$-profile of a right perfect ring. The study of rings in terms of their ($i$ or $p$-)profile was inspired by the study of rings with no right ($i$ or $p$-)middle class, initiated in recent papers by Er, L\'opez-Permouth and S\"okmez, and by Holston, L\'opez-Permouth and Orhan-Ertas. In this paper, we obtain further results about these rings and also we use our results to privede a characterization of a special class of QF rings in which the injectiviry and projectivity domains of all modules coincide.
\end{abstract}

\section{Introduction and preliminaries}

Throughout, $R$ will denote an associative ring with identity, and modules will be unital right modules, unless otherwise stated. As usual, we denote by $\Mod\text{-}R$ the category of right $R$-modules. Recall that a module $M$ is said to be $N$-injective (or injective relative to $N$) if for every submodule $K \leq N$ and every morphism $\varphi: K \rightarrow M$ there exists a morphism $\overline{\varphi}: N \rightarrow M$ such that $\overline{\varphi}|_{K} = \varphi$. For a module $M$, its injectivity domain is defined to be the collection of modules $N$ such that $M$ is $N$-injective, that is $\mathfrak{In}^{-1}(M) = \{N \in \Mod\text{-}R : M \; \text{is} \; N\text{-injective}\}$. For convenience, we will say that a class of modules ${\mathcal P}$ is an (injective) portfolio if there exists a module $M$ such that $\mathcal{P} = \mathfrak{In}^{-1}(M)$. It is well-known (e.g. \cite[Proposition 16.13]{anderson}) that injectivity domains are closed under submodules, homomorphic images and arbitrary direct sums. A module $M$ is called injective if $\mathfrak{In}^{-1}(M) = \Mod\text{-}R$. On the other hand, Alahmadi, Alkan and L\'opez-Permouth defined in \cite{alahmadi} the concept of a poor module, namely, a module $N$ with smallest possible injectivity domain, that is, $\mathfrak{In}^{-1}(N)$ consists precisely of the semisimple modules. In \cite[Proposition 1]{er}, Er, L\'opez-Permouth and S\"okmez proved that every ring has poor modules.



Similarly, a module $M$ is said to be $N$-projective if for every epimorphism $g: N \rightarrow K$ and every morphism $\psi:M \rightarrow K$, there exists a morphism $\overline{\psi}: M \rightarrow N$ with $\psi = g\overline{\psi}$. The projectivity domain of $M$ is then defined as $\pdom^{-1}(M) = \{N \in \Mod\text{-}R : M \; \text{is} \; N\text{-projective}\}$. Projectivity domains are closed under submodules, quotients and finite direct sums (\cite[Proposition 16.12]{anderson}). If $M$ has a projective cover, then $\pdom^{-1}(M)$ is also closed under arbitrary direct products (\cite[Exercise 17.16]{anderson}). Clearly, a module $M$ is projective if $\pdom^{-1}(M) = \Mod\text{-}R$. An opposite notion to projectivity was considered by Holston, L\'opez-Permouth and Orhan in \cite{holston}, as they studied $p$-poor modules. A module $M$ is $p$-poor if $\pdom^{-1}(M)$ contains only semisimple modules. In that paper, the authors proved that every ring has (semisimple) $p$-poor modules(\cite[Theorem 2.8]{holston}). 



A class of modules is said to be a {\it hereditary pretorsion class} if it is closed under homomorphic images, submodules and arbitrary direct sums. Hereditary pretorsion classes play a central role in torsion theory; they also appear in a different setting. Wisbauer's book \cite{wisbauer} documents the movement he lead to generalize the objectives, methods and results of module theory (seen as the study of the category Mod-$R$) by considering, for every module $M$, the full subcategory $\sigma[M]$ of Mod-$R$ having as objects all modules subgenerated by $M$ (i.e. all submodules of homomorphic images of direct sums of copies of $M$). Wisbauer's program has been very popular and has been pursued by many authors (c.f. \cite{dung}, \cite{violapriori1}, etc.) It is not uncommon to use $\sigma[M]$ also to denote only the objects in that category. We like saying, as a recognition to Wisbauer's contributions that any such class of modules is a Wisbauer class. Every Wisbauer class is a hereditary pretorsion class and, conversely, for every hereditary pretorsion class ${\mathcal{T}}$, if we refer to the direct sum of a complete set of non-isomorphic cyclic modules in ${\mathcal{T}}$, then ${\mathcal{T}}$ is precisely the Wisbauer class $\sigma[M]$. Consequently we refer to hereditary pretorsion classes as Wisbauer classes. Furthermore, we denote the class of all Wisbauer classes over a ring R as Wis-$R$.Note that $\wis\text{-}R$ has a natural lattice structure, with the partial order given by inclusion. Also note that every Wisbauer class is completely determined by the cyclic modules in it. Thus, $\wis\text{-}R$ is in bijective correspondence with a set.

 
 A subfunctor of the identity functor $\tau: \Mod\text{-}R \rightarrow \Mod\text{-}R$ is called a {\it left exact preradical} if, for $N \leq M$ we have $\tau(N) = \tau(M)\cap N$. The class of left exact preradicals has a natural lattice structure, given by $\tau \leq \eta$ if $\tau(N) \leq \eta(N)$ for all $N \in \Mod\text{-}R$, $\tau\wedge\eta(N) = \tau(N)\cap\eta(N)$ for all $N \in \Mod\text{-}R$, and $\tau\vee\eta = \bigwedge_{\rho\geq \eta, \tau} \rho$. We denote this lattice by $\lep\text{-}R$.
 
 A set of right ideals $\mathfrak{F}$ is called a {\it linear filter of right ideals} if it satisfies the following axioms: F1) $R \in \mathfrak{F}$; F2) $I, J \in \mathfrak{F}$ implies that $I \cap J \in \mathfrak{F}$; F3) If $I \in \mathfrak{F}$ and $I \leq J$, then $J \in \mathfrak{F}$; and F4) $(I:r) \in \mathfrak{F}$ for all $I \in \mathfrak{F}$ and $r \in R$. For a two-sided ideal $I$ of $R$, the set $\eta(I) = \{J : I \leq J\}$ is a linear filter, and a linear filter is of this form if and only if it is closed under taking arbitrary intersections (\cite[Proposition 1.14]{golan}).  The set of linear filters of right ideals has a natural lattice structure, given by $\mathfrak{F} \leq \mathfrak{G}$ if $\mathfrak{F} \subseteq \mathfrak{G}$, $\mathfrak{F}\wedge\mathfrak{G} = \mathfrak{F}\cap\mathfrak{G}$ and $\mathfrak{F}\vee\mathfrak{G} = \bigwedge_{\mathfrak{H} \geq \mathfrak{F}, \mathfrak{G}}\mathfrak{H}$. We denote this lattice by $\fil\text{-}R$.
 
It is a well-known torsion theoretic fact that the above-mentioned notions are equivalent (see, for example \cite[Chapter VI]{stenstrom}). In \cite{raggi1}, Raggi, R\'{i}os-Montes, Rinc\'on and Fern\'andez-Alonso extended the list of isomorphic lattices by showing that indeed they are all isomorphic to the lattice of fully invariant submodules of a specific type of injective module. They define the concept of a {\it main injective module} as an injective module $\overline{E}$ such that every left exact preradical $\tau$ is of the form $\omega^{\overline{E}}_{\tau(\overline{E})}$, where $\omega^{\overline{E}}_{\tau(\overline{E})}(N) = \bigcap\{f^{-1}(\tau(\overline{E})) : f \in \Hom(N, \overline{E})\}$. In \cite[Theorem 2.1]{raggi1}, the authors also prove that every ring indeed has a main injective module.

For later reference, we summarize the above mentioned results in the following proposition.

\begin{proposition}\label{torsiontheory}
For all rings $R$, the following lattices are equivalent.
\begin{enumerate}
\item The lattice $\wis\text{-}R$ of Wisbauer classes in $\Mod\text{-}R$.
\item The lattice $\fil\text{-}R$ of linear filters of right ideals of $R$.
\item The lattice $\lep\text{-}R$ of left exact preradicals in $\Mod\text{-}R$.
\item The lattice $\mathcal{S}_{fi}(\overline{E})$ of fully invariant submodules of any main injective right $R$-module $\overline{E}$.
\end{enumerate}
\end{proposition}
 
 
 We denote by $E(M), J(M), \Soc(M)$ and $Z(M)$ the injective hull, Jacobson radical, socle and singular submodule of a module $M$, respectively. If $N$ and $M$ are modules, we denote the trace of $N$ in $M$ by  $\Tr_N(M) := \sum\{f(N) : f \in \Hom_R(N,M)\}$. If ${\cal A}$ is a class of modules, we denote the trace of ${\cal A}$ in $M$ by $\Tr_{\cal A}(M) = \sum_{N \in {\cal A}}\Tr_N(M)$. We denote by $\ssmod$-$R$ the full subcategory of $\Mod\text{-}R$ generated by the semisimple modules. Recall that a module $M$ is said to be quasi-injective if it is injective relative to itself, that $R$ is said to be a QI-ring if every quasi-injective module is injective, and that $R$ is said to be a QF-ring if every projective module is injective or, equivalently, if $R$ is a right noetherian right self-injective ring.
 

\section{The injective profile of a ring}

\begin{definition}
Let $R$ be a ring. We call a class $\mathcal{A}$ of modules an $i$-portfolio if there exists $M \in \Mod\text{-}R$ such that $\mathcal{A} = \idom^{-1}(M)$. The class $\{\mathcal{A} \subseteq \Mod\text{-}R : \mathcal{A} \; \text{is an} \; i\text{-portfolio}\}$ is called the right injective profile (right $i$-profile, for short) of $R$ and we denote it by $i\mathcal{P}_{r}(R)$. Similarly, we define the left $i$-profile of $R$ and denote it by $i\mathcal{P}_{\ell}(R)$. When there's no confusion, we denote the right $i$-profile of $R$ just by $i\mathcal{P}(R)$.
\end{definition}

Note that every $i$-portfolio is a Wisbauer class, as it is closed under submodules, quotient modules and arbitrary direct sums. Then, $i\mathcal{P}(R)$ is in bijective correspondence with a set, so we will think of it as a set. Our first goal is to give an intrinsic  description of an $i$-portfolio. We start with the following lemma, that tells us that $i\mathcal{P}(R)$ is closed under arbitrary intersections.

\begin{lemma}\label{intersection}
Let $R$ be a ring. Let $\mathbb{X} \subseteq i\mathcal{P}(R)$. Then, $\bigcap\mathbb{X}$ is an $i$-portfolio.
\end{lemma}
\begin{proof}
Note that we can think of $\mathbb{X}$ as a set. For every $\mathcal{A} \in \mathbb{X}$, let $M_{\mathcal{A}}$ be a module such that $\mathcal{A} = \idom^{-1}(M_{\mathcal{A}})$. It is then easy to see that $\bigcap\mathbb{X} = \idom^{-1}(\prod_{\mathcal{A} \in \mathbb{X}}M_{\mathcal{A}})$.
\end{proof} \\

Note that, as a consequence of Lemma \ref{intersection}, $i\mathcal{P}(R)$ is a complete lattice and is, in fact, a sublattice of $\wis\text{-}R$. Moreover, since every module is injective with respect to any semisimple module, $i\mathcal{P}(R)$ is a sublattice of the interval $[\ssmod\text{-}R, \Mod\text{-}R] \subseteq \wis\text{-}R$.

\begin{proposition}\label{product2}
Let $R$ and $S$ be rings. Then we have a lattice isomorphism $i\mathcal{P}(R \times S) \cong i\mathcal{P}(R) \times i\mathcal{P}(S)$.
\end{proposition}
\begin{proof}
If $M \in \Mod\text{-}(R\times S)$, then $M = M_R \times M_S$, with $M_R \in \Mod\text{-}R$ and $M_S \in \Mod\text{-}S$. Note that, in this case, $\mathfrak{In}^{-1}(M_{(R\times S)}) = \mathfrak{In}^{-1}(M_R) \times \mathfrak{In}^{-1}(M_S)$. From here, the isomorphism is clear.
\end{proof} \\

One may conjecture that if the injective profile of a ring $R$ may be decomposed as the product of two nontrivial lattices then a decomposition of the ring exists that explains the phenomena. That is, however, not the case. See Example \ref{quiver}.

As a consequence of Proposition \ref{product2}, we have that if $R$ is any ring and $S$ is a semisimple artinian ring, then $i\mathcal{P}(R) \cong i\mathcal{P}(R\times S)$. The following proposition is clear.

\begin{proposition}
Let $R$ and $S$ be Morita-equivalent rings. Then, $i\mathcal{P}(R) \cong i\mathcal{P}(S)$.
\end{proposition}

\begin{example}\label{zeta}
We calculate $i\mathcal{P}(\mathbb{Z})$. Let $\mathcal{W}$ be a Wisbauer class containing the semisimple modules, and suppose $\mathcal{W} \not= \Mod\text{-}\mathbb{Z}$. Then, for every prime number $p$ there exists $\alpha_p \in \mathbb{Z}^{+}\cup\{\infty\}$ such that $\mathcal{W} = \sigma[\bigoplus_{p \; \text{prime}} \mathbb{Z}_{p^{\alpha_p}}]$. Note that, if $\alpha_p = \infty$ for every $p$, then $\mathcal{W}$ is the class of torsion modules and, in this case, $\mathcal{W} = \mathfrak{In}^{-1}(\mathbb{Z})$. If not, notice that $\mathcal{W} = \mathfrak{In}^{-1}(\bigoplus_{p \; prime} \mathbb{Z}_{p^{\alpha_p}})$. Finally, if $\mathcal{W} = \Mod\text{-}\mathbb{Z}$, then $\mathcal{W} = \mathfrak{In}^{-1}(\mathbb{Q})$. So, in each case, $\mathcal{W}$ is an $i$-portfolio. Thus, $i\mathcal{P}(\mathbb{Z}) = [\ssmod\text{-}\mathbb{Z}, \Mod\text{-}\mathbb{Z}]$.
\end{example} 

As we shall see, the conclusion of Example \ref{zeta}, namely that any Wisbauer class containing the semisimples is a portfolio, is not particular of the integers, but it holds for any ring. To prove this, we use the following notion.

\begin{definition}
Let $M$ and $N$ be modules. We say that $M$ rises to $N$, and write $M \uparrow N$, if every $M$-injective module is $N$-injective.
\end{definition}

Note that, by definition, if $N \in \sigma[M]$ then $M \uparrow N$. Also, if $M \uparrow N$ and $N \uparrow K$, then $M \uparrow K$. In particular, if $M \uparrow N$ and $K \leq N$, then $M \uparrow K$. \\

We show that, under a rather reasonable hypothesis, the condition $M \uparrow N$ is actually equivalent to $N \in \sigma[M]$. Note that if $M$ is any module and $N$ is a semisimple module, then $M \uparrow N$, so we cannot expect the implication $M \uparrow N \Rightarrow N \in \sigma[M]$ to hold unless $M$ subgenerates every semisimple module. As it turns out, this condition is indeed sufficient for the equivalence. The following lemma is a building block towards that conclusion. 


\begin{lemma}\label{notzero}
Let $M$ and $N$ be modules, and assume $M \uparrow N$. Then, either $N$ is semisimple or $\Tr_{\sigma[M]}(N) \not= 0$.
\end{lemma}
\begin{proof}
Assume that $N$ is not semisimple, and assume, on the contrary, that $\Tr_{\sigma[M]}(N) = 0$. Let $K \leq N$ be a non-direct summand. Then, $K$ is not $N$-injective. Since $\Tr_{\sigma[M]}(K) = 0$, $K$ is $M$-injective, a contradiction with our assumption. Then, $\Tr_{\sigma[M]}(N) \not= 0$. 
\end{proof}



\begin{theorem}\label{theorem1}
Let $R$ be any ring, and let $M$ be a module that subgenerates every semisimple module. Then, for any module $N$, $M \uparrow N$ if and only if $N \in \sigma[M]$.
\end{theorem}
\begin{proof}
$[\Leftarrow]$ is clear. For $[\Rightarrow]$, assume that $M$ subgenerates every semisimple module and that $M \uparrow N$. If $N$ is semisimple, we are done. If not, we show first that $\Tr_{\sigma[M]}(N)$ must be essential in $N$. Indeed, let $T \leq N$ be a nonzero submodule. Then, $T\cap\Tr_{\sigma[M]}(N) = \Tr_{\sigma[M]}(T).$ If $T$ is semisimple, $\Tr_{\sigma[M]}(T) = T$. If $T$ is not semisimple then, by Lemma \ref{notzero} and noting that $M \uparrow T$, we have that $\Tr_{\sigma[M]}(T) \not= 0$. Then, $\Tr_{\sigma[M]}(N)$ is essential in $N$. This implies that $E(\Tr_{\sigma[M]}(N)) = E(N)$. Now, $\Tr_{\sigma[M]}(N) \leq \Tr_{\sigma[M]}(E(N)) \leq E(N)$, which implies that $E(\Tr_{\sigma[M]}(E(N))) = E(N)$. Note that $\Tr_{\sigma[M]}(E(N))$ is $M$-injective. By our assumptions, it is $N$-injective. Hence, every morphism $N \rightarrow E(\Tr_{\sigma[M]}(E(N))) = E(N)$ has its image in $\Tr_{\sigma[M]}(E(N))$. In particular, if we take the inclusion morphism we have that $N \leq \Tr_{\sigma[M]}(E(N))$. Hence, $N \in \sigma[M]$.
\end{proof}

\begin{theorem}\label{every}
Let $R$ be a ring, and let $\mathcal{W}$ be a Wisbauer class in $\Mod\text{-}R$ such that $\ssmod\text{-}R \subseteq \mathcal{W}$. Then, $\mathcal{W}$ is an $i$-portfolio, that is, there exists a module $M$ such that $\mathfrak{In}^{-1}(M) = \mathcal{W}$. In other words, the following lattices are the same:
\begin{enumerate}
\item $i\mathcal{P}(R)$.
\item The interval $[\ssmod\text{-}R, \Mod\text{-}R] \subseteq \wis\text{-}R$.
\end{enumerate}
And, consequently, the following lattices are isomorphic.
\begin{enumerate}
\item$i\mathcal{P}(R)$.
\item The lattice of linear filters of right ideals $\mathfrak{F}$ such that $I \in \mathfrak{F}$ for any maximal right ideal $I$.
\item The lattice of left exact preradicals $\tau$ such that $\Soc \leq \tau$. 
\item The lattice of fully invariant submodules $M$ of any main injective module $\overline{E}$ such that $\Soc(\overline{E}) \subseteq M$.
\end{enumerate}
\end{theorem}
\begin{proof}
Let $N$ be a module such that $\mathcal{W} = \sigma[N]$. Let $\mathbb{X}$ be the collection of portfolios $\mathcal{A}$ such that $N \in \mathcal{A}$. We claim that $\sigma[N] = \bigcap\mathbb{X}$. It is clear that $\sigma[N] \subseteq \bigcap\mathbb{X}$. Not let $K$ be a module not in $\sigma[N]$. Since $N$ subgenerates every semisimple module, Theorem \ref{theorem1} implies that there exists a portfolio $\mathcal{B}$ with $N \in \mathcal{B}$ and $K \not\in \mathcal{B}$. Then, $\sigma[N] = \bigcap\mathbb{X}$. By Lemma \ref{intersection}, $\sigma[N]$ is a portfolio. Then, the lattices $i\mathcal{P}(R)$ and $[\ssmod\text{-}R, \Mod\text{-}R] \subseteq \wis\text{-}R$ are isomorphic. The second result is a refinement of Proposition \ref{torsiontheory}.
\end{proof}

\begin{corollary}\label{coatomic}
For any ring $R$, the lattice $i\mathcal{P}(R)$ is modular and coatomic.
\end{corollary}
\begin{proof}
The lattice (4) in the list of isomorphic lattices of Theorem \ref{every} is clearly modular. By Zorn's lemma, the lattice (2) in the same list is coatomic.
\end{proof}


Note that $i\mathcal{P}(\mathbb{Z})$ has only one coatom, namely the class of torsion modules, which is the injectivity domain of $\mathbb{Z}$. As a first application of Theorem \ref{every}, we show that this is indeed the case for all right uniform rings.

\begin{proposition}
Let $R$ be a right uniform ring. Then, $i\mathcal{P}(R)$ has only one coatom, namely the class of singular modules, which is the injectivity domain of any non-injective nonsingular module.
\end{proposition}
\begin{proof}
Note that, in this case, the linear filter corresponding to the Wisbauer class of singular modules is $\{I \leq R : I \not= 0\}$. From here, the first assertion is clear, as this linear filter is maximum in $\fil\text{-}R$. Now let $M$ be a non-injective nonsingular module. Since for every singular module $N$ we have that Hom$(N,M) = 0$, every singular module is in the injectivity domain of $M$. Then, $\mathfrak{In}^{-1}(M)$ is precisely the class of singular modules.
\end{proof} \\

We saw in Corollary \ref{coatomic} that $i\mathcal{P}(R)$ is a modular and coatomic lattice. If we set other conditions on $R$, we get a nicer lattice. Recall that a ring $R$ is said to be a right QI-ring if every quasi-injective right $R$-module is injective (\cite{boyle}). Since they were introduced, QI-rings have played a central role in ring theory. The following proposition tells us that QI-rings have a particularly well-behaved $i$-profile. Note that any QI-ring is necessarily a right noetherian right $V$-ring.

\begin{proposition}
Let $R$ be a $QI$-ring. Then, $i\mathcal{P}(R)$ is distributive.
\end{proposition}
\begin{proof}
Let $\{E_i\}_{i \in I}$ be a set of representatives of isomorphism classes of indecomposable injective right $R$-modules, and let $E = \bigoplus_{i \in I}E_i$. By \cite[Remark 2.7]{raggi1}, $E$ is a main injective module. Now, let $I' \subseteq I$ be such that $\{E_i\}_{i \in I'}$ is a set of representatives of isomorphism classes of simple modules. Let $J = I \setminus I'$. Let $E_1 = \bigoplus_{i \in I'} E_i$, $E_2 = \bigoplus_{i \in J} E_i$. Then, $\Soc(E) = E_1$ and $E = E_1 \oplus E_2$. Now, if $A \leq E$ is a fully invariant submodule, then $A = A_1 \oplus A_2$, with $A_1$ fully invariant in $E_1$ and $A_2$ fully invariant in $E_2$. Then, the set of fully invariant submodules of $E$ that contain $\Soc(E)$ is in bijective correspondence with the set of fully invariant submodules of $E_2$. Now, if $M$ is a fully invariant submodule of $E_2$, then $M = \bigoplus_{i \in J} M_i$, with $M_i$ fully invariant in $E_i$ for all $i \in J$. Since $R$ is a QI-ring and $E_i$ is indecomposable injective for all $i \in J$, the only fully invariant submodules of $E_i$ are $0$ and $E_i$. For any fully invariant submodule $M = \bigoplus_{i \in J}M_i$, let $\varphi(M) = \{j \in J : M_j = E_j\} \subseteq J$. Note that, for $M, N$ fully invariant submodules of $E_2$, $\varphi(M + N) = \varphi(M) \cup \varphi(N)$ and $\varphi(M\cap N) = \varphi(M)\cap\varphi(N)$. Then, $i{\cal P}(R)$ is isomorphic to a sublattice of $2^J$, so, in particular, it is a distributive lattice.
\end{proof}

\begin{proposition}\label{artinian}
Let $R$ be a right artinian ring. Then, $i\mathcal{P}(R)$ is anti-isomorphic to the lattice of ideals contained in $J(R)$. In particular, $i\mathcal{P}(R)$ is an artinian and noetherian lattice. Thus, $i\mathcal{P}(R)$ is also atomic.
\end{proposition}
\begin{proof}
Since $R$ is right artinian, every linear filter of right ideals of $R$ is closed under arbitrary intersections. Then, (\cite[Proposition 1.14]{golan}) for every linear filter $\mathfrak{F}$ there exists a two-sided ideal $I$ of $R$ such that $\mathfrak{F} = \eta(I) = \{J : I \leq J\}$. Since $i\mathcal{P}(R)$ is isomorphic to the lattice of linear filters of right ideals that contain every maximal right ideal, it follows that  $i\mathcal{P}(R)$ is anti-isomorphic to the lattice of ideals contained in $J(R)$.
\end{proof} \\

From Proposition \ref{artinian}, we see the following.

\begin{corollary}\label{art}
Let $R$ be a right artinian ring. Then, $R$ has no right $i$-middle class if and only if $J(R)$ contains no nontrivial ideals of $R$.
\end{corollary}

\begin{corollary}
Let $R$ be an artinian ring. Then, $i\mathcal{P}_{r}(R) \cong i\mathcal{P}_{\ell}(R)$. In particular, $R$ has no right $i$-middle class if and only if it has no left $i$-middle class.
\end{corollary}

\begin{example}
Let $R$ be an artinian chain ring with composition length $\ell(R)$. Then, $i\mathcal{P}(R)$ is a linearly ordered lattice of length $\ell(R) - 1$.
\end{example}

\begin{example}\label{qfrings}
Let $R$ be a QF-ring. By \cite[Proposition 2.2]{raggi2}, $R$ is main injective as a right module over itself. Then, $i\mathcal{P}(R)$ is isomorphic to the lattice of ideals of $R/\Soc(R)$. Note that, in this case, $R$ has no (left or right) $i$-middle class if and only if $R/\Soc(R)$ is a simple artinian ring.
\end{example}

\section{Rings without injective middle class and rings with linearly ordered injective profile.}

It is clear that a ring is semisimple artinian if and only if $i\mathcal{P}(R)$ is a singleton. Rings whose $i$-profile consists of two elements (necessarily linearly ordered) are studied in \cite{er}, where the authors name them rings without injective middle class. The natural next step is to enquire about the structure of rings for which $i\mathcal{P}(R)$ is a chain. As we will see, this class of rings includes, among others, right chain rings.

If a ring has no right $i$-middle class then every non-semisimple quasi injective right module is injective. This notion is generalization of a right QI-ring (every quasi-injective right module is injective (cf \cite{boyle}). In \cite[Proposition 8]{er} it is shown that for a right SI-ring $R$ with homogeneous and essential right socle, $R$ has no right i-middle class if and only if every nonsemisimple quasi-injective right module is injective. The following proposition gets the same equivalence with a much weaker hypothesis.

\begin{proposition}
Let $R$ be a right semiartinian ring. Then, the following conditions are equivalent.
\begin{enumerate}
\item $R$ has no right $i$-middle class.
\item Every non-semisimple quasi-injective right $R$-module is injective.
\end{enumerate}
\end{proposition}
\begin{proof}
$(1) \Rightarrow (2)$ is clear, even in the case where $R$ is not right semiartinian. For $(2) \Rightarrow (1)$, let $\overline{E}$ be a main injective $R$-module, and let $M$ be a fully invariant submodule of $\overline{E}$ such that $\Soc(\overline{E}) < M$. Now the injective hull of $M$, $E(M)$ is a direct summand of $\overline{E}$, that is, there exists a module $E_2$ such that $\overline{E} = E(M) \oplus E_2$. Now, $\Soc(E_2) = \Soc(\overline{E})\cap E_2 \leq M\cap E_2 = 0$. Since $R$ is right semiartinian, $E_2 = 0$. Then, $\overline{E} = E(M)$. But now $M$ is fully invariant in its injective hull, so it is quasi-injective. But $M$ is not semisimple, which implies that $M$ is injective, so $M = \overline{E}$. Then, $\overline{E}$ has only two fully invariant submodules that contain $\Soc(\overline{E})$ (namely $\Soc(\overline{E})$ and $\overline{E}$ itself). This implies that $R$ has no right $i$-middle class.
\end{proof} \\

The preceding Proposition, together with Proposition \ref{art} has the following consequence.

\begin{corollary}
Let $R$ be a right artinian ring. Then, the following conditions are equivalent:
\begin{enumerate}
\item $R$ has no right $i$-middle class.
\item Every non-semisimple quasi-injective right $R$-module is injective.
\item $J(R)$ contains no nonzero two-sided ideals of $R$.
\end{enumerate}

Since condition 3 is left-right symmetric, we have that, for an artinian ring, every non-semisimple quasi-injective right $R$-module is injective if and only if every non-semisimple quasi-injective left $R$-module is injective.
\end{corollary}

Note that we obtain that a (non-semisimple) right artinian ring with no right $i$-middle class has to have Loewy length 2. \\

Rings with linearly ordered $i$-profile are a natural generalization of rings with no right $i$-middle class. We start with a couple of easy propositions on the module structure of rings with linearly ordered $i$-profile.

\begin{proposition}\label{linear}
Let $R$ be a ring for which $i{\mathcal P}(R)$ is linearly ordered. Let $\overline{E}$ be a main injective module, and let $M$ be a fully invariant submodule of $\overline{E}$ such that $\Soc(\overline{E}) \leq M$. Then, $\Soc(\overline{E}/M)$ is either zero or homogeneous.
\end{proposition}
\begin{proof}
Assume there exists a fully invariant submodule $\Soc(\overline{E}) \leq M \leq \overline{E}$  such that $\Soc(\overline{E}/M)$ is not homogeneous. This gives us two fully invariant submodules of $\overline{E}$ that contain $M$ which are not comparable. Then, $i{\cal P}(R)$ is not linearly ordered.
\end{proof}

\begin{proposition}\label{linear2}
Let $R$ be a ring such that $i{\mathcal P}(R)$ is linearly ordered, and let $I, J \leq J(R)$. Then, either $R/I \in \sigma[R/J]$ or $R/J \in \sigma[R/I]$.
\end{proposition}
\begin{proof}
Since $I, J \leq J(R)$, both $\sigma[R/I]$ and $\sigma[R/J]$ contain every semisimple module. Then, either $\sigma[R/I] \subseteq \sigma[R/J]$ or viceversa. The result follows.
\end{proof} \\

In \cite{violapriori1}, the problem of characterizing those modules $M$ for which the lattice of Wisbauer classes in $\sigma[M]$ is linearly ordered is studied. Note that this condition is too restrictive for our study. For example, if $R$ is any semisimple artinian ring with two non-isomorphic simple modules $S_1$, $S_2$, then the lattice of Wisbauer classes in $\Mod\text{-}R$ is not linearly ordered, as $\sigma[S_1]$ and $\sigma[S_2]$ are not comparable, but clearly $i\mathcal{P}(R)$ is a chain. However, it is clear that every ring with linearly ordered lattice of Wisbauer classes has linearly ordered $i$-profile. The following is an immediate consequence of \cite[Proposition 2.1]{violapriori1}.

\begin{proposition}\label{uniserial}
Let $R$ be a right uniserial ring. Then, $i\mathcal{P}(R)$ is linearly ordered.
\end{proposition}

If $R$ is a local ring, then every non-trivial (that is, different from $\{R\}$) linear filter of right ideals contains every maximal ideal. Thus, in this case, \cite[Proposition 2.2]{violapriori1} gives us a necessary and sufficient condition for $R$ to have linearly ordered $i$-profile.

\begin{proposition}\label{noethlocal}
Let $R$ be a local ring. Then, $i\mathcal{P}(R)$ is linearly ordered if and only if the lattice of ideals of $R$ is linearly ordered.
\end{proposition}

Note that, by Proposition \ref{artinian}, if $R$ is a right artinian ring with linearly ordered profile, then $i\mathcal{P}(R)$ is finite, as it is of finite length. The next proposition shows that this is not true for arbitrary (even noetherian) rings.

\begin{proposition}\label{noetherianchain}
Let $R$ be a noetherian chain ring which is not artinian. Then, $i\mathcal{P}(R) \cong \omega + 2$. In particular, $i\mathcal{P}(R)$ is linearly ordered but not finite.
\end{proposition}
\begin{proof}
By \cite[Proposition 5.3]{facchini}, $R$ is a duo ring and the lattice of ideals of $R$ is $R > J(R) > J^2(R) > \dots$, where $J^{i}(R) \not= 0$ for all $i \in \mathbb{N}$ and $\bigcap_{i \in \mathbb{N}} J^{i}(R) = 0$. $i\mathcal{P}(R)$ is isomorphic to the lattice of linear filters of ideals $\mathfrak{F}$ such that $J(R) \in \mathfrak{F}$. Let $\mathfrak{F}$ be such a linear filter. If the set $ A_{\mathfrak{F}} = \{ n \in \mathbb{Z}^{+} : J^n(R) \in \mathfrak{F}\}$ has a maximum element $m$, then $\mathfrak{F} = \mathfrak{F}_m = \{J^k(R) : 0 \leq k \leq m\}$. If $A_{\mathfrak{F}}$ does not have a maximum element, then either $\mathfrak{F} = \mathfrak{F}_{\infty} = \{J^k(R) : k \in \mathbb{N}\}$ or $\mathfrak{F} = \mathfrak{A}$, the set of all ideals in $R$. Then, all linear filters that contain $J(R)$ are $\mathfrak{F}_1 < \mathfrak{F}_2 < \cdots < \mathfrak{F}_{\infty} < \mathfrak{A}$, so $i\mathcal{P}(R) \cong \omega + 2$. 
\end{proof}

In light of Proposition \ref{noetherianchain} it is worth mentioning that if $i\mathcal{P}(R)$ is an ordinal, it cannot be a limit ordinal, as $i\mathcal{P}(R)$ is coatomic. In fact, for the same reason we see that if $R$ is not semisimple artinian and $i\mathcal{P}(R)$ is an ordinal, then $i\mathcal{P}(R) \cong \alpha + 2$ for some ordinal $\alpha$. \\

The following proposition, together with its corollary, gives us a necessary condition for a ring to have a finite and linearly ordered profile. Note that, proceeding exactly as in the proof of \cite[Proposition 2]{er}, we get that if $\Gamma$ is a complete and irredundant set of representatives of cyclic right $R$-modules and if $M = \prod_{N \in \Gamma} N$, then $M$ is an $i$-poor module.

\begin{proposition}\label{cyclics}
Let $R$ be a non-semisimple artinian ring such that $i\mathcal{P}(R)$ is linearly ordered. Then, for every non-poor module $K$, there exists a cyclic module $C$ such that $\mathfrak{In}^{-1}(C) \subsetneq \mathfrak{In}^{-1}(K)$.
\end{proposition}
\begin{proof}
Let $M$ and $\Gamma$ be as in the preceding paragraph, so $M$ is an $i$-poor module. Let $K$ be a non-$i$-poor module and assume that for every $N \in \Gamma$, $\mathfrak{In}^{-1}(N) \not\subsetneq \mathfrak{In}^{-1}(K)$. Since $i{\cal P}(R)$ is linearly ordered, this implies that for every $N \in \Gamma$; $\mathfrak{In}^{-1}(K) \subseteq \mathfrak{In}^{-1}(N)$. Then, $\ssmod$-$R \subsetneq \mathfrak{In}^{-1}(K) \subseteq \bigcap_{N \in \Gamma}\mathfrak{In}^{-1}(N) = \mathfrak{In}^{-1}(M)$, a contradiction. Then, there exists $C \in \Gamma$ such that $\mathfrak{In}^{-1}(C) \subsetneq \mathfrak{In}^{-1}(K)$.
\end{proof}

\begin{corollary}\label{cyclicpoor}
Let $R$ be a ring such that $i\mathcal{P}(R)$ is linearly ordered and atomic (for example, if $i\mathcal{P}(R)$ is finite or if $R$ is as in Example \ref{noetherianchain}). Then, $R$ has a cyclic right module that is $i$-poor.
\end{corollary}
\begin{proof}
In Proposition \ref{cyclics}, put $K$ as a module such that $\mathfrak{In}^{-1}(K)$ is an atom of ${\cal P}(R)$.
\end{proof} \\

Note that for a right artinian ring, Corollary \ref{cyclicpoor} does not yield anything new, as it is shown in \cite[Theorem 3.3]{alahmadi} that for a right artinian ring it is always the case that the cyclic module $R/J(R)$ is $i$-poor. However, the next result tells us that any right artinian ring with linearly ordered $i$-profile necessarily has a simple $i$-poor module. Note that not every ring has a cyclic $i$-poor module. For example, the ring of integers $\mathbb{Z}$ does not have one.

\begin{proposition}\label{simpleipoor}
Let $R$ be a right artinian ring such that $i\mathcal{P}(R)$ is linearly ordered. Then, $R$ has a simple $i$-poor module.
\end{proposition}
\begin{proof}
Since $R$ is right artinian, $R/J(R)$ is a semisimple $i$-poor module. Say $R/J(R) = \bigoplus_{i = 1}^{k}S_i$. Note that in this case $\mathfrak{In}^{-1}(\bigoplus_{i=1}^{k}S_i) = \bigcap_{i = 1}^{k} \mathfrak{In}^{-1}(S_i)$. Now proceed as in Proposition \ref{cyclics} and Corollary \ref{cyclicpoor}.
\end{proof}\\

Recall that for two preradicals $\tau$ and $\lambda$, the preradical $(\tau:\lambda)$ is defined in such a way that $(\tau:\lambda)M/\tau(M) = \lambda(M/\tau(M))$, and that a preradical is called a radical if $(\tau:\tau) = \tau$ (or, equivalently, if $\tau(M/\tau(M)) = 0$ for all $M \in \Mod$-$R$). If $\tau$ and $\lambda$ are supposed to be left exact preradicals, then $(\tau:\lambda)$ is left exact \cite[Exercise VI.1]{stenstrom}. From the definition, it is clear that $\tau, \lambda \leq (\tau:\lambda)$ for any preradicals. Using these observations, we can prove the following result.

\begin{proposition}\label{loewy}
Let $R$ be a ring with no right $i$-middle class. Then, eiter $\Soc$ is a radical or $R$ is a right semiartinian ring with Loewy length 2.
\end{proposition}
\begin{proof}
Assume $\Soc$ is not a radical. Then, $(\Soc:\Soc)$ is a left exact preradical which is strictly greater than $\Soc$. Since $R$ has no right middle class, we must have $(\Soc:\Soc) = 1_{\Mod\text{-}R}$. Then, $(\Soc:\Soc)(R) = R$ and therefore $R$ is a right semiartinian ring with Loewy length 2.
\end{proof} \\

Note that both possibilities in Proposition \ref{loewy} can happen. For example, if $R$ is Cozzen's ring of differential polynomials in \cite{cozzens} then $R$ is a noetherian ring with no right $i$-middle class which is not right artinian, and in this case $\Soc$ is a radical, as the socle of any right module $M$ splits in $M$. Also note that the requirement of $\Soc$ being a radical is equivalent to the requirement that $\ext^{1}(S, T) = 0$ for any semisimple modules $S$ and $T$.

\section{The projective profile of a ring.}

\begin{definition}
Let ${\cal A}$ be a class of modules. We say that ${\cal A}$ is a $p$-portfolio if there exists a module $M$ such that ${\cal A} = \pdom^{-1}(M)$. The class $\{{\cal A} \subseteq \Mod\text{-}R : {\cal A} \; \text{is a} \; p\text{-portfolio}\}$ is called the right $p$-profile of $R$ and we denote it by $p{\cal P}_{r}(R)$. Similarly, we define the left $p$-profile of $R$ and denote it by $p{\cal P}_{\ell}(R)$. When there's no confusion, we denote $p{\cal P}_{r}(R)$ just by $p{\cal P}(R)$.
\end{definition} 

In general, the $i$-profile of a ring better behaved than its $p$-profile. For instance, we have seen that the $i$-profile of any ring is a set, but this need not be true for the $p$-profile, as we will see next. We call a module $M$ a test-module for projectivity if any module $N$ is projective whenever it is $M$-projective. While the famous Baer criterion tells us that the ring $R$ is always a test-module for injectivity, test-modules for projectivity need not exist. Even in the category of $\mathbb{Z}$-modules, the existence of a test-module for projectivity is equivalent to the existence of a Whitehead group, and this problem has been shown to be undecidable in ZFC (\cite{shelah}). The existence of test-modules for projectivity is important in this problem because of the following proposition.

\begin{theorem}\label{notset}
Let $R$ be a ring. If $p{\cal P}(R)$ is a set, then there exists a test-module for projectivity.
\end{theorem}
\begin{proof}
We show that if there is no test-module for projectivity, then $p{\cal P}(R)$ is a proper class. For brevity, if $\lambda$ is a cardinal we say that a module $M$ is $\lambda$-projective if it is $R^{(\lambda)}$-projective. We construct an injective relation $f: \text{ORD} \rightarrow p{\cal P}(R)$ recursively, as follows. Since $R$ is not a test-module for projectivity, let $M(0)$ be a module which is $R$-projective but not projective, and define $f(0) = \pdom^{-1}(M(0))$. Now let $\beta$ be an ordinal:

\begin{enumerate}
\item If $\beta = \alpha + 1$ and $f(\alpha)$ has already been defined as $\pdom^{-1}(M(\alpha))$, where $M(\alpha)$ is a module which is $\lambda(\alpha)$-projective but not projective for some cardinal $\lambda(\alpha)$, let $\lambda(\beta)$ be the least cardinal such that $M(\alpha)$ is not $\lambda(\beta)$-projective. Let $M(\beta)$ be a module which is $\lambda(\beta)$-projective but not projective, and define $f(\beta) = \pdom^{-1}(M(\beta))$.

\item If $\beta$ is a limit ordinal and $f(\gamma)$ has been defined for all $\gamma < \beta$ as $\pdom^{-1}(M(\gamma))$, where $M(\gamma)$ is a module which is $\lambda(\gamma)$-projective but not projective. Then, for every $\gamma < \beta$ there exists a cardinal $\kappa(\gamma)$ such that $M(\gamma)$ is not $\kappa(\gamma)$-projective. Let $\lambda(\beta) = \bigcup_{\lambda<\beta}\kappa(\gamma)$ and let $M(\beta)$ me a module which is $\lambda(\beta)$-projective but not projective. Define $f(\beta) = \pdom^{-1}(M(\beta))$.
\end{enumerate}

Then, $f$ defines an injective function from the class of ordinals to $p{\cal P}(R)$. We conclude that $p{\cal P}(R)$ is a proper class.
\end{proof}

\begin{corollary}
There exists a model of ZFC in which the $p$-profile of $\mathbb{Z}$ is not a set.
\end{corollary}

In light of Theorem \ref{notset}, we cannot use a similar argument to that of Lemma \ref{intersection} to show that $p\mathcal{P}(R)$ is closed under arbitrary intersections, so, contrary to the injective case, we only know that $p\mathcal{P}(R)$ is a semilattice with first and last element. A few results do echo their injective counterparts. The following two results are proved exactly like in the injective case, so we omit their proofs.

\begin{proposition}
Let $R$ and $S$ be rings. Then, $p\mathcal{P}(R \times S) \cong p\mathcal{P}(R) \times p\mathcal{P}(S)$. In particular, if $S$ is semisimple artinian then $p\mathcal{P}(R \times S) \cong p\mathcal{P}(R)$.
\end{proposition}

Note that, although ring decompositions induce decompositions in the injective and projective profile, the converse is not true, as Example \ref{quiver} shows.

\begin{proposition}
Let $R$ and $S$ be Morita-equivalent rings. Then, $p\mathcal{P}(R) \cong p\mathcal{P}(S)$.
\end{proposition}

Although Theorem \ref{notset} tells us that the study of the projective profile of a ring is, in general, harder than that of the injective profile, we have some results in the case where projectivity is well-behaved. Namely, if $R$ is a right perfect ring, then any module has a projective cover, so the projectivity domain of any module is closed under arbitrary direct products. In this case, it is clear that $p\mathcal{P}(R)$ is a set, and we actually have that $p\mathcal{P}(R)$ is anti-isomorphic to the lattice of two-sided ideals contained in $J(R)$. Even more, for any element in $p\mathcal{P}(R)$, there is an easy way to find a module $M$ such that the aforementioned element coincides with $\mathfrak{Pr}^{-1}(M)$.

\begin{lemma}
Let $R$ be a right perfect ring, and let $I \subseteq J(R)$ be an ideal of $R$. Then, $\Mod\text{-}R/I$ is a $p$-portfolio, and in fact, $\Mod\text{-}R/I = \mathfrak{Pr}^{-1}(R/I)$.
\end{lemma}
\begin{proof}
Since $R$ is a right perfect ring, $\mathfrak{Pr}^{-1}(R/I)$ is a Wisbauer class closed under products (\cite[Exercise 17.16]{anderson}), so it is of the form $\Mod\text{-}R/L$ for some ideal $L$ of $R$ (\cite[Proposition 1.14]{golan}). It is clear that $\Mod\text{-}R/I \subseteq \mathfrak{Pr}^{-1}(R/I)$ so we have that $L \leq I$. If we assume that $L \not= I$, then, as $I$ is contained in $J(R)$, $I/L$ is superfluous in $R/L$ and so $R/L$ and $R/I$ would be two non-isomorphic projective covers of $R/I$ in $\Mod\text{-}R/L$, a contradiction. Then, $\mathfrak{Pr}^{-1}(R/I) = \Mod\text{-}R/I$.
\end{proof}

\begin{proposition}\label{perfect2}
Let $R$ be a right perfect ring. Then, $p\mathcal{P}(R)$ is anti-isomorphic to the lattice of ideals of $R$ that are contained in $J(R)$.
\end{proposition}
\begin{proof}
Since $R$ is right perfect, every $p$-portfolio is a Wisbauer class closed under products. Then, its corresponding linear filter is closed under arbitrary intersections, so it is of the form $\eta(I) := \{J : I \leq J\}$ for a two-sided ideal $I \leq R$ (\cite[Proposition 1.14]{golan}). Thus, $p\mathcal{P}(R)$ can be identified with a subset of the lattice of ideals of $R$ that are contained in $J(R)$. By the preceding lemma, this identification is surjective.
\end{proof}

\begin{example}\label{quiver}
Let $K$ be an algebraically closed field and $Q$ be the quiver $2 \buildrel \alpha \over \leftarrow 1 \buildrel \beta \over \rightarrow 3$. Let $A := KQ$, the path algebra of $Q$. $A$ is right artinian, so $i\mathcal{P}(A)$ and $p\mathcal{P}(A)$ are both anti-isomorphic to the lattice of ideals contained in $J(A)$. Note that $J(A) = \langle \alpha, \beta \rangle$ and that the ideals it contains are $0, \langle \alpha \rangle, \langle \beta \rangle$ and $J(A)$. Then, $i\mathcal{P}(A)$ and $p{\cal P}(A)$ are isomorphic to the product of two intervals of length $1$, $i\mathcal{P}(A) \cong p{\cal P}(A) \cong {\cal L}_1 \times {\cal L}_1$. Since $Q$ is a connected quiver, $A$ is indecomposable as a ring. This shows that non-trivial factorizations on the (injective or projective) profile of a ring do not induce factorizations of the ring itself.
\end{example}

\begin{proposition}
Let $R$ be a right perfect ring. Then, the following conditions are equivalent. \\

(1) $R$ has no right p-middle class. \\

(2) Every non-semisimple quasi-projective right $R$-module is projective. \\

(3) $J(R)$ contains no non-trivial two-sided ideals.
\end{proposition}
\begin{proof}
(1) $\Leftrightarrow$ (3) is by Proposition \ref{perfect2} and (1) $\Rightarrow$ (2) is clear. We show (2) $\Rightarrow$ (3). Assume (2) and suppose, on the contrary, that there exists a two-sided ideal $I$, $0 \not= I \subsetneq J(R)$. Then, $R/I$ is not semisimple and it's clearly quasi-projective, so it must be projective. But this is a contradiction, as $I$ is a superfluous ideal of $R$. Then, $J(R)$ contains no non-trivial two-sided ideals and we're done.
\end{proof}

\begin{corollary}\label{isomorphiclat}
Let $R$ be a right artinian ring. Then, the lattices $i\mathcal{P}_{r}(R)$, $p\mathcal{P}_{r}(R)$ and $p\mathcal{P}_{\ell}(R)$ are all isomorphic. In particular, we have that for a right artinian ring the following conditions are equivalent.
\begin{enumerate}
\item $R$ has no right $i$-middle class.
\item Every non-semisimple quasi-injective right $R$-module is injective.
\item $R$ has no right $p$-middle class.
\item Every non-semisimple quasi-projective right $R$-module is projective.
\item $R$ has no left $p$-middle class.
\item Every non-semisimple quasi-projective left $R$-module is projective.
\end{enumerate}
\end{corollary}

The following proposition tells us that if, instead of assuming that $R$ is a right artinian ring we assume that $R$ is a right perfect ring, we still get $(1) \Rightarrow (3)$ in the preceding corollary.

\begin{proposition}
Let $R$ be a right perfect ring. If $R$ has no right $i$-middle class, then $R$ has no right $p$-middle class.
\end{proposition}
\begin{proof}
We may assume that $R$ is not semisimple artinian. Then, as $R$ is right perfect, this implies that $J(R) \not= 0$. We show that $J(R)$ does not contain properly a nonzero two-sided ideal of $R$. If this were the case, say there exists $I$ a two-sided ideal with $0 < I < J(R)$ then we have three different linear filters of right ideals containing all the maximal right ideals, namely $\eta(0)$, $\eta(I)$ and $\eta(J(R))$. This contradicts the fact that $R$ has no right $i$-middle class. Thus, the only ideals contained in $J(R)$ are $0$ and $J(R)$. This implies, using again that $R$ is right perfect, that $R$ has no projective middle class.
\end{proof} \\

We end this section with results that are similar to Proposition \ref{cyclics} and Corollary \ref{cyclicpoor}. Note that it is shown in \cite[Theorem 2.8]{holston} that, for any ring, if $\Gamma$ is a set of representatives of isomorphism classes of simple modules, then $\bigoplus_{S \in \Gamma} S$ is $p$-poor.

\begin{proposition}
Let $R$ be a non-semisimple artinian ring such that $p\mathcal{P}(R)$ is linearly ordered. Then, for every non-$p$-poor module $M$, there exists a simple module $S$ with $\mathfrak{Pr}^{-1}(S) \subsetneq \mathfrak{Pr}^{-1}(M)$.
\end{proposition}

\begin{corollary}
Let $R$ be a ring such that $p\mathcal{P}(R)$ is linearly ordered and atomic (for example, an artinian chain ring). Then, $R$ has a simple $p$-poor module.
\end{corollary}

Note that, if $R$ is a two sided artinian ring, then $i\mathcal{P}_{\ell}(R) \cong p\mathcal{P}_{\ell}(R) \cong i\mathcal{P}_{r}(R) \cong p\mathcal{P}_{\ell}(R)$. For this reason, we use $\mathcal{P}(R)$ to denote any of these (isomorphic) lattices, and we call any of these lattices the profile of $R$.

\begin{proposition}
Let $R$ be a QF-ring with a unique simple (up to isomorphism). Then, $\mathcal{P}(R)$ is linearly ordered of length $n$ if and only if $\mathcal{P}(R/\Soc(R))$ is linearly ordered of length $n-1$.
\end{proposition}
\begin{proof}
Under the hypothesis, $\Soc(R)$ is the trace of the unique simple module. As $R$ is self-injective, this implies that $\Soc(R)$ is the unique minimal ideal of $R$. Now, we can decompose $R$ into a direct sum of indecomposables, $R = e_1R \oplus \cdots \oplus e_mR$, with $\text{top}(e_iR) \cong \text{top}(e_jR)$ for every $i,j$. Since $R$ is right artinian, this implies that $e_iR \cong e_jR$ for every $i, j$. Since $J(e_iR)$ is the unique maximal submodule of $e_iR$ and $R$ is self-injective, this implies that $J(R)$ is the unique maximal ideal of $R$. \\
Now assume $\mathcal{P}(R)$ is linearly ordered of length $n$. Then, by the characterization of the profile of an artinian ring, the lattice of two-sided ideals contained in $J(R)$ is $0 < \Soc(R) < I_3 \dots < I_n = J(R)$. This implies that the lattice of two-sided ideals of $R/\Soc(R)$ contained in $J(R/\Soc(R))$ is $0 \leq I_3/\Soc(R) \dots < I_n/\Soc(R) = J(R/\Soc(R)$. Hence, $\mathcal{P}(R/\Soc(R))$ is linearly ordered of length $n-1$. \\
On the other hand, if the profile of $R/\Soc(R)$ is linearly ordered of length $n-1$, the bijective correspondence theorem implies that the lattice of ideals contained in $J(R)$ is linearly ordered of length $n$ (note that here we are using very heavily the fact that $\Soc(R)$ is the minimal ideal of $R$.
\end{proof}

\begin{proposition}\label{matrices}
Let $R := \left(\begin{array}{cc} \mathbb{Q} & 0 \\ \mathbb{R} & \mathbb{R} \end{array}\right)$. Then, $R$ has no left $i$-middle class, has no (right or left) $p$-middle class, but it is not without right $i$-middle class.
\end{proposition}
\begin{proof}
It is well known that $R$ is left artinian (hence two-sided perfect) but not right artinian. The Jacobson radical of $R$ is $J(R) = \left(\begin{array}{cc} 0 & 0 \\ \mathbb{R} & 0 \end{array}\right)$, which does not contain nonzero two-sided ideals. This shows that $R$ has no left $i$-middle class and that it has no right and left $p$-middle class. Now we show that $R$ is not without right $i$-middle class. To see this, it suffices to show three distinct linear filters of right ideals containing all the maximal right ideals. Clearly, two of these linear filters are the trivial filter $\eta(0)$ consisting of all right ideals, and the filter $\eta(J(R))$, consisting of all right ideals containing $J(R)$.\\
Now let $\mathcal{A} = \{\left(\begin{array}{cc} 0 & 0 \\ X & 0 \end{array}\right) : X \; \text{is a} \; \mathbb{Q}\text{-subspace of} \; \mathbb{R} \; \text{such that} \; \dim_{\mathbb{Q}}(\mathbb{R}/X) < \infty\}$, and $\mathcal{B} = \{\left(\begin{array}{cc} \mathbb{Q} & 0 \\ X & 0 \end{array}\right) : X \; \text{is a} \; \mathbb{Q}\text{-subspace of} \; \mathbb{R} \; \text{such that} \; \dim_{\mathbb{Q}}(\mathbb{R}/X) < \infty\}$ We claim that $\mathfrak{F} := \mathcal{A}\cup\mathcal{B}\cup\eta(J(R))$ is a linear filter of right ideals. It is clear that axioms F1), F2) and F3) of the definition of a linear filter are satisfied by the family $\mathfrak{F}$. We show F4), namely, that for every $I \in \mathfrak{F}$, $x \in R$, $(I:x) = \{y \in R : xy \in I\}$ is again a member of $\mathfrak{F}$. Note that, if $I \in \eta(J(R))$ then $(I:x) \in \eta(J(R)) \subseteq \mathfrak{F}$, as $\eta(J(R))$ is a linear filter of right ideals. Also, if $I \in \mathcal{B}$ then there exists $J \in \mathcal{A}$ with $J \subseteq I$, so $(J:x) \subseteq (I:x)$. Hence, we may assume without loss of generality that $I \in \mathcal{A}$, that is, there exists a subspace of $\mathbb{R}$ of finite codimension, $Y$, such that $I = \left(\begin{array}{cc} 0 & 0 \\ Y & 0 \end{array}\right)$.\\
Let  $x = \left(\begin{array}{cc}q & 0 \\ r_1 & r_2\end{array}\right) \in R$, $y = \left(\begin{array}{cc} Q & 0 \\ R_1 & R_2 \end{array}\right) \in (I:x)$. Then, $Qq = 0$, $R_2r_2 = 0$ and $qr_1 + R_1r_2 \in Y$. We consider four cases.\\
{\bf Case 1.} $q \not= 0$, $r_2 \not= 0$. In this case, $Q = 0$, $R_2 = 0$ and $R_1r_2 \in Y$. Hence, $(I:x) = \left(\begin{array}{cc} 0 & 0 \\ r_{2}^{-1}Y & 0 \end{array}\right) \in \mathfrak{F}$.\\
{\bf Case 2.} $q \not= 0$, $r_2 = 0$. In this case, $Q = 0$ and $R_2$, $R_1$ can be any real number. Hence, $(I:x) = \left(\begin{array}{cc} 0 & 0 \\ \mathbb{R} & \mathbb{R} \end{array}\right) \in \mathfrak{F}$. \\
{\bf Case 3.} $q = 0, r_2 = 0$. Then, we have two subcases. If $r_1 \in Y$, then $x \in I$ so $(I:x) = R$. If $r_1 \not\in Y$, then $(I:x) = \left(\begin{array}{cc} 0 & 0 \\ \mathbb{R} & \mathbb{R} \end{array}\right)$. In any case, $(I:x) \in \mathfrak{F}$. \\
{\bf Case 4.} $q = 0, r_2 \not= 0$. Note that, in this case, $(I:x)$ contains the right ideal $\left(\begin{array}{cc}0 & 0 \\ r_2^{-1}Y & 0\end{array}\right) \in \mathfrak{F}$. Since $\mathfrak{F}$ is closed under superideals, $(I:x) \in \mathfrak{F}$. \\

Hence, $\mathfrak{F}$ is a linear filter of right ideals. Since $\eta(J(R)) \subseteq \mathfrak{F}$, $\mathfrak{F}$ contains all the maximal right ideals. From the definition of $\mathfrak{F}$, it is clear that $\eta(J(R)) \subsetneq \mathfrak{F} \subsetneq \eta(0)$. Hence, $R$ is not without right $i$-middle class. \end{proof}


\section{When projectivity and injectivity domains coincide.}

A classic characterization of QF rings is that they are those rings for which every projective module is injective, or, equivalently, every injective module is projective. That inspires the following definition.

\begin{definition}
Let $R$ be a ring. We say that $R$ is a right super QF ring if $\idom^{-1}(M) = \pdom^{-1}(M)$ for all right modules $M$.
\end{definition}

Clearly, any right super QF ring is QF, as the definition implies that every projective module is injective. Super QF rings have their origin in \cite[Proposition 3.14]{holston}, where the authors show that any QF-ring with homogeneous right socle and $J(R)^2 = 0$ is super QF. Now we show that, unlike the class of QF-rings, the class of super QF-rings is closed under quotient rings.

\begin{proposition}\label{quotient}
Let $R$ be a right super QF ring, and let $I$ be an ideal of $R$. Then, $R/I$ is right super QF.
\end{proposition}
\begin{proof}
We may identify Mod-$R/I$ with the full subcategory of Mod-$R$ consisting of modules which are annihilated by $I$. Then, for any $R/I$-module $M$, $\idom^{-1}(M_{R/I}) = \idom^{-1}(M_R)\cap$Mod-$R/I = \pdom^{-1}(M_R)\cap$Mod-$R/I = \pdom^{-1}(M_{R/I})$. Therefore, $R/I$ is a right super QF ring.
\end{proof}

\begin{proposition}\label{product}
Let $R_1$ and $R_2$ be right super QF rings. Then, $R_1 \times R_2$ is right super QF.
\end{proposition}
\begin{proof}
Let $M$ be a $R_1 \times R_2$-module. Then, $M = M_1 \oplus M_2$, where $M_1 = MR_1 \in \Mod\text{-}R_1$ and $M_2 = MR_2 \in \Mod\text{-}R_2$. Then, $\idom^{-1}(M) = \idom^{-1}(M_{1_{R_1}}) \times \idom^{-1}(M_{2_{R_2}}) = \pdom^{-1}(M_{1_{R_1}}) \times \pdom^{-1}(M_{2_{R_2}}) = \pdom^{-1}(M)$. Then, $R_1\times R_2$ is right super QF.
\end{proof}

Next, we show that any artinian chain ring is right super QF. Surprisingly, these are, basically, the only examples of super QF rings.

\begin{proposition}\label{chain}
Let $R$ be an artinian chain ring. Then, $R$ is right and left super QF.
\end{proposition}
\begin{proof}
We show that $R$ is right super QF. Since $R$ is artinian and every right module is the direct sum of cyclic modules, it suffices to show that $\idom^{-1}(R/I) = \pdom^{-1}(R/I)$ for all ideals $I$. Since $R$ is artinian chain, the ideals of $R$ are all of the form $J^n$, $n \leq \ell$, where $\ell$ denotes the Loewy length of $R$. The cyclic modules in $\idom^{-1}(R/J^n)$ are $R/J^m$, with $m \leq n$, and these are precisely the cyclic modules in $\pdom^{-1}(R/J^n)$: note that $R/J^n$ is QF as a ring, so the cyclic modules $R/J^m$ ($m \leq n$) are in $\idom^{-1}(R/J^n)$ and $\pdom^{-1}(R/J^n)$. If $k < m$, then there exists a non-split epimorphism $R/J^m \rightarrow R/J^k$ and a non-split monomorphism $R/J^k \rightarrow R/J^m$, this shows that $R/J^k$ is not in $\pdom^{-1}(R/J^n)$ or $\idom^{-1}(R/J^n)$. Then, $R$ is right super QF.
\end{proof}

\begin{proposition}\label{morita}
Let $R$ be a right super QF ring and let $S$ be Morita equivalent to $R$. Then, $S$ is right super QF.
\end{proposition}
\begin{proof}
Follows from the fact that if $\Phi : \Mod\text{-}R \rightarrow \Mod\text{-}S$ is an equivalence of categories, then $A \in \idom^{-1}(M)$ ($A \in \pdom^{-1}(M)$) if and only if $\Phi(A) \in \idom^{-1}(\Phi(M))$ ($\Phi(A) \in \pdom^{-1}(\Phi(M))$, respectively).
\end{proof}

Recall that a ring $R$ is said to be right FGF if every finitely generated right $R$-module embeds in a free module. It is clear that every QF ring is right and left FGF. Whether every right FGF ring is necessarily QF is an open problem (see \cite{faith} for more references). It is interesting that the condition \lq all factor rings of $R$ are QF\rq \; is equivalent to \lq all factor rings of $R$ are FGF\rq \; (\cite[Theorem 6.1]{faith}). The next theorem tells us that these rings are precisely the (right or left) super QF rings.

\begin{theorem}\label{theoremsuperqf}
Let $R$ be a ring. Then, the following conditions are equivalent: \\

1) $R$ is right super QF. \\

1') $R$ is left super QF. \\

2) Every factor ring of $R$ is right super QF. \\

2') Every factor ring of $R$ is left super QF. \\

3) Every factor ring of $R$ is QF. \\

4) Every factor ring of $R$ is FGF. \\

5) $R$ is isomorphic to a product of full matrix rings over artinian chain rings.

\end{theorem}
\begin{proof}
1) $\Rightarrow$ 2) is Proposition \ref{quotient}, 2) $\Rightarrow$ 3) is clear, and 3) $\Leftrightarrow$ 4) $\Leftrightarrow$ 5) is by \cite[Theorem 6.1]{faith}. For 5) $\Rightarrow$ 1), note that if $R \cong \prod_{i = 1}^{k}\mathbb{M}_{n_i}(D_i)$, where the $D_i$'s are artinian chain rings, then every $\mathbb{M}_{n_i}(D_i)$ is right super QF, as it is Morita equivalent to $D_i$ and this is right super QF by Proposition \ref{chain}. By Proposition \ref{product}, $R$ is right super QF. Then, 1), 2), 3) and 4) are equivalent. By left-right-simmetry, 1'), 2'), 3), 4) and 5) are also equivalent.
\end{proof} \\

Note that if a QF-ring has no (injective or projective) middle class then, by Corollary \ref{isomorphiclat}, it is super QF. 

\begin{proposition}
Let $R$ be a QF ring. Then, the following are equivalent:

\begin{enumerate}
\item $R/\Soc(R)$ is simple artinian.
\item $J(R)$ contains no non-trivial two-sided ideals.
\item $J(R)$ is semisimple homogeneous.
\item $R \cong S \times T$, with $S \cong \mathbb{M}_n(D)$ for an artinian chain ring of length $2$ $D$, and $T$ is $0$ or semisimple artinian.
\end{enumerate}
\end{proposition}
\begin{proof}
(1) $\Leftrightarrow$ (2) is from Corollary \ref{art} and Example \ref{qfrings}.  (2) $\Rightarrow$ (4) is from the fact that if (2) is satisfied, then $R$ is super QF, as it is QF and has no i-middle class and no p-middle class. (4) $\Rightarrow$ (2) is true because, if (4) holds, then the injective profile of $R$ is isomorphic to the injective profile of $D$, which has no i-middle class. (2) $\Rightarrow$ (3) follows because $J(R)\cap\Soc(R) \not= 0$, so $J(R) \subseteq \Soc(R)$ and, since it contains no non-trivial two-sided ideals, it must be homogeneous. Finally, (3) $\Rightarrow$ (2) follows from the fact that $R$ is self-injective, so the trace of any simple module in $R$ is a minimal ideal of $R$.
\end{proof}


\begin{thebibliography}{99}
\bibitem{alahmadi} A.N. Alahmadi, M. Alkan, S.R. L\'opez-Permouth. \lq\lq Poor modules: The opposite of injectivity\rq\rq. Glasg. Math. J. 52A(2010) 7-17
\bibitem{anderson} F.W. Anderson and K.R. Fuller. {\it Rings and categories of modules. $2^{\text{nd}}$ edition.} Springer-Verlag, New York, 1991.
\bibitem{boyle} A.K. Boyle. \lq\lq Hereditary QI-rings\rq\rq. Trans. Amer. Math. Soc. 192 (1974) 115-120.
\bibitem{dung} N.V. Dung, D.V. Huynh, P.F. Smith, R. Wisbauer. {\it Extending modules.} Pitman Res. Notes in Math. Ser., vol. 313, 1994.
\bibitem{cozzens} J.H. Cozzens. \lq\lq Homological properties of the ring of differential polynomials\rq\rq. Bull. Amer. Math. Soc. 76 (1) (1970) 75-79
\bibitem{er}N. Er, S.R. L\'opez-Permouth, N. S\"okmez. \lq \lq Rings whose modules have maximal or minimal injectivity domain\rq\rq. J. Algebra 330 (1) (2011) 404-417
\bibitem{facchini} A. Facchini. {\it Module theory: endomorphism rings and direct sum decompositions in some classes of modules}. Birkh\"auser Verlag, 1998.
\bibitem{faith} C. Faith, D.V. Huynh. \lq\lq When self-injective rings are QF: A report on a problem\rq\rq J. Algebra Appl. 1 (1) (2002) 75-105
\bibitem{golan} J. Golan {\it Linear topologies on a ring: an overview}. Pitman Research Notes in Mathematics Series 159. 1987.
\bibitem{holston} C. Holston, S.R. L\'opez-Permouth, N. Orhan Erta\c s, \lq\lq Rings whose modules have maximal or minimal projectivity domain\rq\rq J. Pure Appl. Algebra 216 (3) (2012) 673-678
\bibitem{mohamed} S.H. Mohamed, B.J. M\"uller. {\it Continuous and discrete modules}. Cambridge University Press, 1990.
\bibitem{raggi1} F. Raggi, J. R\'{i}os, H. Rinc\'on, R. Fern\'andez-Alonso. \lq\lq Basic preradicals and main injective modules\rq\rq. J. Algebra Appl. 8 (1) (2009) 1-16
\bibitem{raggi2} F. Raggi, J. R\'{i}os, H. Rinc\'on, R. Fern\'andez-Alonso. \lq\lq Main injective rings\rq\rq, Comm. Alg. 39:4 (2011) 1226-1233
\bibitem{shelah} S. Shelah. \lq\lq Infinite abelian groups, Whitehead problem and some constructions\lq\lq, Israel J. Math. 28 (3) (1974).
\bibitem{stenstrom} B. Stenstr\"om. {\it Rings of quotients}. Springer-Verlag. Berlin, 1975.
\bibitem{violapriori1} A.M. de Viola-Priori, J. E. Viola-Priori, R. Wisbauer. \lq\lq Module categories with linearly ordered closed subcategories\rq\rq, Comm. Alg. 22:9 (1994) 3613-3627
\bibitem{wisbauer} R. Wisbauer. {\it Grundlagen der Modul- und Ringtheorie.} Verlag Reinhard Fischer, Munich, 1988.
\end{thebibliography}
\end{document}